\documentclass[11pt]{amsart}

\usepackage{amsmath,amsthm,amssymb,fancyhdr,graphicx,bbm,cancel,color,mathrsfs,todonotes,hyperref,xypic, pinlabel}
\usepackage[all]{xy}

\newtheorem{thm}{Theorem}[section]

\newtheorem{prop}[thm]{Proposition}
\newtheorem{cor}[thm]{Corollary}
\theoremstyle{definition} 

\theoremstyle{remark} 

\newenvironment{defn}[1][Definition.]{\begin{trivlist}
\item[\hskip \labelsep {\bfseries #1}]}{\end{trivlist}}

\newcommand{\dmo}{\DeclareMathOperator}

\newcommand{\R}{\mathbb{R}}
\newcommand{\Q}{\mathbb{Q}}
\newcommand{\co}{\mathbb{C}}\newcommand{\Z}{\mathbb{Z}}

\newcommand{\al}{\alpha}\newcommand{\be}{\beta}\newcommand{\ga}{\gamma}\newcommand{\ep}{\epsilon}\newcommand{\si}{\sigma}\newcommand{\ze}{\zeta}
\newcommand{\vp}{\varphi}\newcommand{\Om}{\Omega}\newcommand{\ka}{\kappa}\newcommand{\la}{\lambda}
\newcommand{\Ga}{\Gamma}
\newcommand{\wtil}{\widetilde}\newcommand{\ta}{\theta}\newcommand{\Si}{\Sigma}\newcommand{\om}{\omega}

\newcommand{\cd}{\cdots}\newcommand{\ld}{\ldots}
\newcommand{\sbs}{\subset}\newcommand{\bs}{\backslash}\newcommand{\pa}{\partial}\newcommand{\es}{\emptyset}\newcommand{\bbm}{\mathbbm}

\newcommand{\xra}{\xrightarrow}

\newcommand{\ra}{\rightarrow}
\newcommand{\hra}{\hookrightarrow}

\newcommand{\bb}[1]{\mathbb{#1}}\newcommand{\ca}[1]{\mathcal{#1}}

\newcommand{\fr}[2]{\frac{#1}{#2}}

\newcommand{\ot}{\otimes}
\newcommand{\lan}{\langle}\newcommand{\ran}{\rangle}
\newcommand{\op}{\oplus}
\newcommand{\ti}{\times}

\dmo{\sgn}{sign}
\dmo{\we}{\wedge}
\dmo{\ind}{ind}\dmo{\Ind}{Ind}
\dmo{\bop}{\bigoplus}\dmo{\pic}{Pic}
\dmo{\coker}{coker}\dmo{\vol}{Vol}\dmo{\gal}{Gal}\dmo{\perm}{Perm}
\dmo{\tor}{Tor}\dmo{\ext}{Ext}\dmo{\Ext}{Ext}
\dmo{\aut}{Aut}
\dmo{\Aut}{Aut}
\dmo{\inn}{Inn}\dmo{\var}{Var}
\dmo{\dep}{depth}\newcommand{\rest}[2]{#1\bigr\vert_{#2}}
\dmo{\ad}{ad}\dmo{\curl}{curl}
\dmo{\hy}{\bb H}\dmo{\Sl}{SL}
\dmo{\SO}{SO}\dmo{\psl}{PSL}
\dmo{\isom}{Isom}\dmo{\Isom}{Isom}
\dmo{\conf}{Conf}
\dmo{\stab}{Stab}\dmo{\Jac}{Jac }
\dmo{\diam}{diam}\dmo{\fix}{Fixed}\dmo{\Fix}{Fix}
\dmo{\injR}{injRad}\dmo{\Ad}{Ad}
\dmo{\esv}{ess-vol}\dmo{\out}{Out}\dmo{\Out}{Out}
\dmo{\nil}{Nil}\dmo{\sol}{Sol}
\dmo{\Div}{div}
\dmo{\SU}{SU}
\dmo{\SP}{SP}
\dmo{\Sp}{Sp}
\dmo{\rk}{rk}
\dmo{\rank}{rank}
\dmo{\psp}{PSp}\dmo{\psu}{PSU}
\dmo{\PU}{PU}\dmo{\pgl}{PGL}
\dmo{\Mod}{Mod}\dmo{\range}{Range}
\dmo{\eu}{eu}\dmo{\mi}{mi}
\dmo{\Log}{Log}\dmo{\supp}{supp}
\dmo{\maps}{Maps}\dmo{\Gr}{Gr}
\dmo{\Pin}{Pin}
\dmo{\Spin}{Spin}\dmo{\Str}{Str}
\dmo{\Sq}{Sq}\dmo{\Symp}{Symp}
\dmo{\pd}{PD}\dmo{\PD}{PD}\dmo{\sig}{Sig}
\dmo{\Set}{Set}\dmo{\Top}{Top}
\dmo{\ev}{ev}\dmo{\St}{St}
\dmo{\Pt}{Pt}\dmo{\pt}{pt}

\dmo{\colim}{colim }\dmo{\Pl}{PL}
\dmo{\String}{String}\dmo{\smear}{smear}

\dmo{\dev}{dev}
\dmo{\met}{Met}\dmo{\contact}{Contact}
\dmo{\teich}{Teich}\dmo{\Teich}{Teich}\dmo{\qi}{QI}
\dmo{\der}{Der}
\dmo{\cl}{Cliff}\dmo{\Cl}{Cl}
\dmo{\Pf}{Pf}
\dmo{\ch}{ch}\dmo{\diag}{diag}
\dmo{\grad}{grad}\dmo{\Char}{char}
\dmo{\spec}{Spec}\dmo{\Arg}{Arg}
\dmo{\rad}{rad}\dmo{\im}{Im}
\dmo{\Hom}{Hom}\dmo{\End}{End}
\dmo{\tr}{tr}\dmo{\id}{Id}
\dmo{\gl}{GL}
\dmo{\sym}{Sym}\dmo{\Sym}{Sym}
\dmo{\com}{Comm}
\dmo{\Lk}{Lk}
\dmo{\Rep}{Rep}
\dmo{\Conf}{Conf}
\dmo{\PConf}{PConf}
\dmo{\Push}{Push}
\dmo{\Cont}{Cont}
\dmo{\sm}{\setminus}
\dmo{\vn}{\varnothing}
\dmo{\disk}{\mathbb D}
\dmo{\Trd}{Trd}\dmo{\Mat}{Mat}
\dmo{\Riem}{Riem}
\dmo{\Diffn}{\Diff_0}\dmo{\diff}{Diff}
\dmo{\Diff}{Diff}\dmo{\homeo}{Homeo}
\dmo{\Homeo}{Homeo}\dmo{\Fr}{Fr}
\dmo{\rot}{rot}\dmo{\Emb}{Emb}
\dmo{\Ham}{Ham}\dmo{\Met}{Met}
\dmo{\Ein}{Ein}\dmo{\CP}{\co P}
\dmo{\Per}{Per}\dmo{\Ric}{Ric}
\newcommand{\C}{\mathbb C}\dmo{\Nrd}{Nrd}
\dmo{\Comp}{Comp}\dmo{\PSC}{PSC}
\dmo{\Cent}{Cent}\dmo{\Orb}{Orb}
\dmo{\aind}{a-ind}\dmo{\tind}{t-ind}
\dmo{\constant}{constant}
\dmo{\Td}{Td}
\dmo{\LMod}{LMod}
\dmo{\SMod}{SMod}
\dmo{\SDiff}{SDiff}
\dmo{\Br}{Br}
\dmo{\csch}{csch}
\dmo{\triv}{triv}
\dmo{\genus}{genus}
\dmo{\Homeq}{Homeq}
\dmo{\PP}{\mathbb{P}}
\dmo{\U}{U}

\usepackage{enumerate,pdfpages,stmaryrd} 
\usepackage[margin=0.8in]{geometry}
\usepackage{minitoc}

\usepackage{tikz}
\usepackage{dsfont}
\usetikzlibrary{matrix}

\begin{document}

\title{Characteristic classes of fiberwise branched surface bundles via arithmetic}

\author{Bena Tshishiku}
\address{Department of Mathematics, Stanford University, Stanford, CA 94305} \email{tshishikub@gmail.com}


\date{\today}


\begin{abstract} 
This paper is about cohomology of mapping class groups from the perspective of arithmetic groups. For a closed surface $S$ of genus $g$, the mapping class group $\Mod(S)$ admits a well-known arithmetic quotient $\Mod(S)\ra\Sp_{2g}(\Z)$, under which the stable cohomology of $\Sp_{2g}(\Z)$ pulls back to the algebra generated by the odd MMM classes of $\Mod(S)$. We extend this example to other arithmetic groups associated to mapping class groups and explore some of the consequences for surface bundles. 

For $G=\Z/m\Z$ and for a regular $G$-cover $S\ra\bar S$ (possibly branched), a finite index subgroup $\Ga<\Mod(\bar S)$ acts on $H_1(S;\Z)$ commuting with the deck group action, thus inducing a homomorphism $\Ga\ra\Sp^G$ to an arithmetic group $\Sp^G<\Sp_{2g}(\Z)$. The induced map $H^*(\Sp^G;\Q)\ra H^*(\Ga;\Q)$ can be understood using index theory. To this end, we describe a families version of the $G$-index theorem for the signature operator and apply this to (i) compute $H^2(\Sp^G;\Q)\ra H^2(\Ga;\Q)$, (ii) re-derive Hirzebruch's formula for signature of a branched cover (in the case of a surface bundle), (iii) compute Toledo invariants of surface group representations to $\SU(p,q)$ arising from Atiyah--Kodaira constructions, and (iv) describe how classes in $H^*(\Sp^G;\Q)$ give equivariant cobordism invariants for surface bundles with a fiberwise $G$ action, following Church--Farb--Thibault.\end{abstract}

\maketitle

\section{Introduction}

The goal of this paper is to widen the bridge connecting arithmetic groups and the cohomology of mapping class groups using the index theory of elliptic operators. 

For a surface $S$ the mapping class group $\Mod(S)$ is the group of diffeomorphisms up to isotopy $\Mod(S)=\pi_0\Diff(S)$. For a finite set $Z\sbs S$, we denote $\Mod(S,Z)=\pi_0\Diff(S,Z)$, where $\Diff(S,Z)$ is the group of diffeomorphisms which fix each $z\in Z$. 

{\it Basic example.} Let $S$ be a closed surface of genus $g$. The mapping class group $\Mod(S)$ acts on $H_1(S;\Z)$ inducing a surjection $\Mod(S)\ra\Sp_{2g}(\Z)$. Using index theory, one can answer the question: What is the induced map on cohomology 
\[\al^*:H^*\big(\Sp_{2g}(\Z);\Q\big)\ra H^*\big(\Mod(S);\Q\big)?\]

To explain the answer, we work with bundles and characteristic classes. Let $M\ra B$ be an $S$-bundle. The associated Hodge bundle is a rank-$g$ complex vector bundle $E\ra B$ classified by a map $B\ra B\U(g)$ that factors as 
\[B\ra B\Mod(S)\ra B\Sp_{2g}(\Z)\ra B\Sp_{2g}(\R)\sim B\U(g).\]
The associated K-theory class $E\in K(B)$ is closely related to the index $\ind(D)$ of a family $D=\{D_b\}_{b\in B}$ of elliptic operators on $S$. If $\bar E$ denotes the conjugate bundle, then $\ind(D)=E-\bar E\in K(B)$; see \cite[\S 6]{asiii} and \cite{asiv}. The index theorem gives a topological way to compute the index: $\ind(D)=\pi_!\big(\si(D)\big)$ where $\si(D)\in K(M)$ is the \emph{symbol class}, and $\pi_!: K(M)\ra K(B)$ is the pushforward in K-theory. After applying the Chern character (and the Grothendieck--Riemann--Roch computation), one arrives at a formula 
\begin{equation}\label{eqn:indexfamilies}
\ch (E-\bar E)=\pi_!\left(\fr{x}{\tanh(x/2)}\right),
\end{equation}
where $x=e(T_\pi M)$ is the Euler class of the vertical tangent bundle. See \cite[\S 6]{asiii} and \cite[Thm 5.1]{asiv}. 

On the one hand, the Chern character $\ch\big(E-\bar E\big)$ is determined by the odd Chern classes of $E$, so it is closely related to $\im(\al^*)$. On the other hand, the right-hand side is a polynomial in the odd MMM classes because $\fr{x}{\tanh(x/2)}$ is polynomial in $x^2$. In fact, in the stable range, the image of $H^*\big(\Sp_{2g}(\Z);\Q\big)\ra H^*\big(\Mod(S);\Q\big)$ is the algebra generated by the odd MMM classes $\{\kappa_{2i-1}: i\ge1\}$. 

{\it A more general setting.} Our story begins with the observation that $\Sp_{2g}(\Z)$ is one of many arithmetic groups related to mapping class groups \cite{looijenga_prym, gllm}. Specifically, for a finite, regular $G$-cover $\mu:S\ra\bar S$ branched over $\bar Z\sbs\bar S$, there is a finite index subgroup $\Ga<\Mod(\bar S,\bar Z)$ that lifts to the centralizer $\Mod^G$ of $G<\Mod(S)$. The group $\Ga$ acts on $H_1(S;\Z)$ by $G$-module maps inducing a homomorphism $\Ga\ra\Mod^G\ra\Sp^G$ to the centralizer of $G<\Sp_{2g}(\Z)$. Our first goal is to derive an index formula that will enable us to study the map 
\[\al ^*:H^*\big(\Sp^G;\Q\big)\ra H^*\big(\Mod^G;\Q\big).\]

{\bf Standing assumption.} For simplicity we will restrict attention to the case (i) $G\simeq\Z/m\Z$, and (ii) $G$ acts trivially on the pre-image of the branched set: $\mu^{-1}(\bar Z)=\fix(G)$. 

To study $\al ^*$, we work on the level of bundles as in the previous example, utilizing the fact that $B\Diff(S)^G\ra B\Mod(S)^G$ is a homotopy equivalence (Earle--Schatz \cite{earle-schatz}). Thus a bundle with monodromy in $\Mod^G$ is an example of what we call an \emph{$(S,G)$-bundle}: 

\begin{defn}
Fix an embedding $i:G\hra\Diff(S)$. Let $\pi: M\ra B$ be an $S$ bundle, and suppose $G$ acts on $M$ preserving each fiber. We call $\pi:M\ra B$ an \emph{$(S,G)$-bundle} if there are charts $B=\bigcup_\al U_\al$ and trivializations $\phi_\al:\pi^{-1}(U_\al)\simeq U_\al\ti S$ that identify the $G$ action on each fiber with $i$.
\end{defn}

The primary known examples of $(S,G)$ bundles are obtained by the fiberwise branched covering constructions of Atiyah--Kodaira \cite{atiyah} and Morita \cite{morita_book}. 

{\it Nota bene.} The action $i:G\hra\Diff(S)$ is part of the data of an $(S,G)$ bundle, although we omit it from the notation. Whenever we talk of $(S,G)$ bundles, we will always assume that the action $i:G\hra\Diff(S)$ has been chosen beforehand. 

We will be interested in the following invariants of an $(S,G)$ bundle $\pi:M\ra B$. 
\begin{enumerate}[(i)]
\item {\it Hodge eigenbundles.} The Hodge bundle $E\ra B$ inherits a $G$ action and splits into eigenbundles $E=\bigoplus_{q^m=1}E_q$. These eigenbundles have Chern classes $c_i(E_q)\in H^{2i}(B;\Q)$. 

\item {\it Euler classes.} Each fixed point $z\in S$ corresponds to a component of the fixed set $M^G$ and defines a section $\si_z:B\ra M$. Associated to this section is an Euler class $e_z\in H^2(B;\Q)$, which is defined by $e_z=\si_z^*(e)$, where $e=e(T_\pi M)$ denotes the Euler class of the vertical tangent bundle $T_\pi M\ra M$. 

\end{enumerate} 

The Chern classes of the Hodge eigenbundles and the Euler classes are related by the following index formula. 

\begin{thm}[{\bf Index formula}]\label{thm:index}
Let $S$ be a closed, oriented surface with an orientation-preserving action of a cyclic group $G$ of order $m$. For concreteness, we fix an identification $G<\C^\ti$. Assume that each point stabilizer is either trivial or equal to $G$. Decompose the fixed set $Z=\sqcup_{j=1}^{m-1} Z_j$ where 
\[Z_j=\{z\in Z: G\text{ acts on $T_zS$ with character } g\mapsto g^j\}.\] Let $\pi:M\ra B$ be an $(S,G)$-bundle. Denote $E=\bop_{q^m=1}E_q$ the decomposition of the Hodge bundle $E\ra B$ into $G$-eigenbundles. For $z\in Z$, let $e_z\in H^2(B;\Q)$ be the corresponding Euler class, and if $Z_j\neq\es$ define $\ep_j=\sum_{z\in Z_j}e_z$. Denoting $\ta_j=\fr{2\pi j}{m}$, for $1\le r\le m-1$, 
\begin{equation}\label{eqn:index1}\sum_{q^m=1}q^r\>\big[\ch(E_q)-\ch(\bar E_{\bar q})\big]=\sum_{\substack{1\le j\le m-1\\Z_j\neq\es}} \coth\left(\fr{\ep_j+i\>r\ta_j}{2}\right).\end{equation}
\end{thm}

{\it Remark.} The assumption on point stabilizers implies that $Z_j\neq\es$ only if $\gcd(j,m)=1$. 

Theorem \ref{thm:index} is a families version of the $G$-index theorem for the signature operator. The left-hand side of (\ref{eqn:index1}) is a families version of the $g$-signature $\sig(g,S)$ of \cite{asiii}. For $g=e^{2\pi ir/m}\neq 1$ this index is computed using the Atiyah--Segal localization theorem \cite{asii} and this gives the right-hand side of (\ref{eqn:index1}). The case $g=1$ is special, where one obtains instead the formula from the families index theorem (\ref{eqn:indexfamilies}). Applying (\ref{eqn:indexfamilies}) and (\ref{eqn:index1}) to the universal example relates the stable cohomology of $\Sp^G$ to the subalgebra of $H^*\big(\Mod^G;\Q\big)$ generated by the odd MMM classes and the Euler classes corresponding to the fixed points. The following theorem describes the image of $\al^*$ in the stable range. 

\begin{thm}\label{thm:main} Let $S$ and $G$ be as in Theorem \ref{thm:index}. Assume the genus of $S/G$ is at least 6. For each $1\le j<m/2$ such that $Z_j\cup Z_{m-j}\neq\es$, define $\eta_j$ to be $\ep_j+\ep_{m-j}$ or $\ep_j$ or $\ep_{m-j}$ according to whether $Z_j$ and $Z_{m-j}$ are both nonempty, $Z_j\neq\es$ and $Z_j=\es$, or $Z_j=\es$ and $Z_{m-j}\neq\es$. 
\begin{enumerate}[(i)]
\item The image of $\al ^*:H^2(\Sp^G;\Q)\ra H^2(\Mod^G;\Q)$ is the subspace spanned by $\ka_1$ and $\{\eta_j: 1\le j<m/2\text{ and } Z_j\cup Z_{m-j}\neq\es\}$. 
\item In the stable range, the image of $H^*\big(\Sp^G;\Q\big)\ra H^*\big(\Mod^G;\Q\big)$ is the algebra generated by $\{\kappa_{2i-1}: i\ge1\}$ and $\{\eta_j\}$. 
\end{enumerate} 
\end{thm}

This has the following corollary for the cohomology of the Torelli subgroup of $\Mod^G$. 

\begin{cor}
Let $\ca I^G<\Mod^G$ be kernel of $\Mod^G\ra\Sp^G$. For each $1\le j<m/2$ such that $Z_j\cup Z_{m-j}\neq\es$, the class $\eta_j$ is in the kernel of $H^2(\Mod^G;\Q)\ra H^2(\ca I^G;\Q)$. 
\end{cor}


\subsection{Applications.} 
In Section \ref{sec:app} we discuss some applications of the index formula: 

\vspace{.1in} 
{\bf Geometric characteristic classes after Church--Farb--Thibault.} Following \cite{cft}, a characteristic class $c\in H^k\big(B\Diff(F)\big)$ is called \emph{geometric with respect to cobordism} if two $F$ bundles $M\ra B^k$ and $M_1\ra B_1^k$ have the same characteristic numbers $c^{\#}(M\ra B)=c^{\#}(M_1\ra B_1)$ whenever the manifolds $M$ and $M_1$ are cobordant. In particular, such a characteristic class is insensitive to the fibering $M\ra B$. In the case of an $(S,G)$-surface bundle $M\ra \Si$ over a surface $\Si$, the classes $\ka_1$ and $\ep_j$ are $G$-cobordism invariants of $M$, so as a consequence of Theorems \ref{thm:index} and \ref{thm:main}, we obtain

\begin{cor}\label{cor:bord}
Fix $G\sbs\C^\ti$ with $|G|=m$. Let $M\ra \Si$ be an $(S,G)$-bundle over a surface. For each $q^m=1$, the characteristic numbers $c_1^{\#}(E_q\ra\Si)$ of the Hodge eigenbundle $E_q\ra \Si$ are $G$-cobordism invariants, i.e.\ they depend only on the $G$-cobordism class of $M$.  
\end{cor}

For example, the standard Atiyah--Kodaira example is a surface bundle $S_6\ra M\ra S_{129}$ with a fiberwise $G=\Z/2\Z$ action. The manifold $M$ also fibers as $S_{321}\ra M\ra S_3$, and we have Hodge eigenbundles 
\[\C^3\ra E_1\ra S_{129}\>\>\>\text{ and }\>\>\>\C^3\ra E_{-1}\ra S_{129}\>\>\>\>\>\>,\>\>\>\>\>\>\C^{104}\ra E_1'\ra S_3\>\>\>\text{ and }\>\>\>\C^{217}\ra E_{-1}'\ra S_3.\]
The corollary says that 
\[\big\lan c_1(E_1),[S_{129}] \big\ran=\big\lan c_1(E_1'), [S_3]\big\ran\>\>\>\text{ and } \>\>\> \big\lan c_1(E_{-1}),[S_{129}] \big\ran=\big\lan c_1(E_{-1}'),[S_3]\big\ran,\]
where $\lan\cdot,\cdot\ran$ denotes the pairing between cohomology and homology.
\vspace{.1in}

{\bf Surface group representations.} Any surface bundle $E\ra\Si$ with monodromy in $\Mod^G$ induces surface group representations $\pi_1(\Si)\ra H$ where $H$ is either $\Sp_{2k}(\R)$ or $\SU(a,b)$. The Toledo invariant of this representation coincides with the Chern class $c_1(E_q)$ of one of the Hodge eigenbundles, so these Toledo invariants may be computed by the index formula. For example, the Atiyah--Kodaira construction for $G=\Z/7\Z$ can be used to produce a representation 
\[\al :\pi_1(\Si_{7^{17}+1})\ra\SU(8,13)\ti\SU(9,12)\ti\SU(10,11).\]
whose Toledo invariants (obtained by projecting to individual factors) are all nonzero, distinct, and can be computed using (\ref{eqn:index1}). See \S \ref{sec:app}. 

\vspace{.1in} 
{\bf Hirzebruch's formula for signature of branched covers.} The index formula can be used to express how the odd MMM classes behave under fiberwise branched covers. 
In the case of $\kappa_1$ this allows us to derive Hirzebruch's formula for the signature of a branched cover. Our derivation of Hirzebruch's formula emphasizes the connection to arithmetic groups. 

\vspace{.1in} 
{\bf Remark on proofs.} Theorem \ref{thm:index} is a version of the $G$-index theorem for families and is obtained by combining the results of \cite{asii,asiii,asiv}. As far as the author knows this does not appear in the literature, although it is surely known to experts (see the last sentence of \cite[\S5]{asiv}). Theorem \ref{thm:main} is proved by computing the stable cohomology of $\Sp^G$ (following Borel), relating this cohomology and Chern classes of the Hodge eigenbundles, and applying the index formula to reduce the problem to the linear algebra of circulant matrices. 

\subsection{Outline of paper.} Section \ref{sec:mod} contains some facts about subgroups of $\Mod(\bar S)$ that lift a cover, and a result of Earle--Schatz that describes the classifying space for $(S,G)$ bundles. The index formula of Theorem \ref{thm:index} is proved in Section \ref{sec:index}. In Section \ref{sec:compute} we prove Theorem \ref{thm:main}, and in Section \ref{sec:app} we discuss the applications mentioned above. 

{\it Acknowledgement.} 
The author would like to thank O.\ Randal-Williams for several useful conversations about index theory and for his interest in this project. The author thanks T.\ Church, S.\ Galatius, N. Salter, and L.\ Starkson for helpful conversations. Thanks also to B.\ Farb and N.\ Salter for comments on a draft of this paper. 

\section{Mapping class groups, subgroups, and surface bundles}\label{sec:mod}


For a closed surface $S$ and any manifold $B$, there is a well-known bijection between the collection of homomorphisms $\pi_1(B)\ra\Mod(S)$ (up to conjugacy) and the collection of $S$-bundles $M\ra B$ (up to bundle isomorphism). In this section we record an equivariant version of this fact that gives a monodromy characterization of $(S,G)$-bundles. Then we recall the definition of \emph{liftable subgroups} $\Mod_\mu(\bar S,\bar Z)<\Mod(\bar S,\bar Z)$ associated to a branched cover $\mu:S\ra\bar S$ with branched set $\bar Z\sbs\bar S$. We show that $\Mod_\mu(\bar S,\bar Z)$ is finite index in $\Mod(\bar S,\bar Z)$ and show that it virtually lifts to $\Mod(S,Z)$. 

\subsection{Classifying $(S,G)$-bundles.} 
One of the miracles in the study of surface bundles is that a surface bundle $B\ra B\Diff(S)$ is determined by its monodromy, i.e.\ by the induced homomorphism 
\[\pi_1(B)\ra\pi_1\big(B\Diff(S)\big)\simeq\pi_0\big(\Diff(S)\big)\equiv\Mod(S).\] For an $(S,G)$-bundle $M\ra B$, the classifying map $B\ra B\Diff(S)$ factors through $B\Diff(S)^G$, where $\Diff(S)^G$ is the centralizer of $G<\Diff(S)$.\footnote{We use this notation because $\Diff(S)^G$ is the fixed point set of $G$ acting on $\Diff(S)$ by conjugation.} In this case, we have a homomorphism
\[\pi_1(B)\ra\pi_0\big(\Diff(S)^G\big)\ra\Mod(S)^G,\]
where $\Mod(S)^G$ is the centralizer of (the image of) $G<\Mod(S)$. The following theorem says that an $(S,G)$ bundle $M\ra B$ is determined by its monodromy $\pi_1(B)\ra\Mod(S)^G$. This theorem is a consequence of Earle--Schatz \cite{earle-schatz}.

\begin{thm}\label{thm:es}
Fix a closed surface $S$ with an action of a finite group $G$. For each manifold $B$, there is a bijection 
\[
\left\{\begin{array}{cc}(S,G)\text{-bundles}\\M\ra B\\\text{up to isomorphism}\end{array}\right\}\leftrightarrow \left\{\begin{array}{cc}\text{homomorphisms}\\\pi_1(B)\ra\Mod(S)^G\\\text{up to conjugacy}\end{array}\right\}
\]
\end{thm}

\begin{proof}
Let $\Diff_0(S)$ denote the path component of the identity in $\Diff(S)$. There is a fiber sequence 
\[\Diff_0(S)\cap\Diff(S)^G\ra\Diff(S)^G\ra\Mod(S)^G\]
(which is also an exact sequence of groups). By \cite[\S5(F)]{earle-schatz}, the topological group $\Diff_0(S)\cap\Diff(S)^G$ is contractible, which implies that $B\Diff(S)^G\ra B\Mod(S)^G$ is a homotopy equivalence, and the theorem follows. 
\end{proof}

{\it Remarks.} 
\begin{enumerate}
\item The theorem of Earle--Schatz says, in particular, that $\Diff(S)^G\cap\Diff_0(S)$ is connected, so if $\phi\in\Diff(S)^G$ is isotopic to the identity, then it has an isotopy through diffeomorphisms that commute with $G$ (compare with \cite{birman-hilden}). Consequently the surjection $\pi_0\big(\Diff(S)^G\big)\ra\Mod(S)^G$ is an isomorphism. 

\item By Theorem \ref{thm:es}, if the monodromy $\pi_1(B)\ra\Mod(S)$ of an $S$-bundle $M\ra B$ factors through $\Mod(S)^G$, then $M\ra B$ has the structure of a $(S,G)$-bundle. Without the theorem of Earle--Schatz, it is not obvious why a bundle with monodromy $\Mod(S)^G$ admits a fiberwise $G$-action. 

\item As a consequence of (the proof of) Theorem \ref{thm:es}, the cohomology $H^*\big(\Mod(S)^G\big)$ can be identified with the ring of characteristic classes of $(S,G)$-bundles. 
\end{enumerate} 

{\it Important remark.} In this note we will be interested in $(S,G)$-bundles that are classified by maps $B\ra B\Diff(S,Z)^G$, where $Z\sbs S$ is the fixed set of $G$. Recall that $\Diff(S,Z)$ denotes the group of diffeomorphisms that restrict to the identity on $Z$ (in general an element of $\Diff(S)^G$ need only preserve $Z$). The above discussion implies that these $(S,G)$-bundles are also classified by their monodromy $\pi_1(B)\ra\Mod(S,Z)^G$. Note that $\Mod(S)^G$ and $\Mod(S,Z)^G$ differ by a finite group. 

\subsection{Liftable subgroups.} For a finite regular $G$-cover $\mu:S\ra\bar S$ branched over $\bar Z$ with $Z=\mu^{-1}(\bar Z)$, define the \emph{liftable subgroup} 
\[\Mod_\mu(\bar S,\bar Z)=\{[f]\in\Mod(\bar S,\bar Z)\mid f \text{ admits a lift }\wtil f\in\Diff(S,Z)^G\}.
\]
By definition, there is an exact sequence 
\begin{equation}\label{eqn:ses}1\ra G\ra\Mod(S,Z)^G\ra\Mod_{\mu}(\bar S,\bar Z)\ra1.\end{equation}
Our goal in this subsection is to explain the following proposition, which is a modification of an argument of Morita \cite[Lemma 4.13]{morita_book} to the case of a branched cover. (Note: although our standing assumption is that $G$ is cyclic, this proposition does not require this.)

\begin{prop}\label{prop:lift}Let $G$ be a finite group, and fix a regular $G$-cover $\mu:S\ra\bar S$ branched over $\bar Z\sbs\bar S$. Assume $S$ is closed and $\chi(\bar S\sm\bar Z)<0$. Then the liftable subgroup $\Mod_\mu(\bar S,\bar Z)<\Mod(\bar S,\bar Z)$ is finite index and contains a finite-index subgroup that maps to $\Mod(S,Z)^G$. 
\end{prop}

\begin{proof}
Remove $\bar Z$ and $Z=\mu^{-1}(\bar Z)$ to get an unbranched regular cover $\hat\Si\ra\Si$ (between noncompact surfaces). Consider the induced exact sequence 
\[1\ra\hat T\ra T\xra{q}G\ra 1,\] where $\hat T=\pi_1(\hat\Si)$ and $T=\pi_1(\Si)$. Give $T$ the standard presentation $T=\lan a_1,b_1,\ld,a_g,b_g,p_1,\ld,p_n\mid\prod[a_i,b_i]\prod p_j=1\ran$, where $p_j$ is a loop around the $j$-th puncture (here $n=|Z|$). By assumption $\chi(\Si)<0$, so we may realize $T$ as a Fuchsian group $T<\psl_2(\R)$ so that the $a_i,b_i$ are hyperbolic and the $p_j$ are parabolic. An automorphism $\phi\in\Aut(T)$ is called \emph{type-preserving} if it preserves hyperbolic (resp.\ parabolic) elements. Denote $\ca A(T)<\Aut(T)$ the group of type-preserving automorphisms. By \cite[Theorem 1]{hm} there are isomorphisms
\[\ca A(T)\simeq\Mod(\Si,*)\>\text{ and }\>\ca A(T)/\inn(T)\simeq\Mod(\Si),\]
where $*\in\Si$ is a basepoint. Similarly, we can define $\ca A(\hat T)$, and we have isomorphism $\ca A(\hat T)\simeq\Mod(\hat\Si,*)$ and $\ca A(\hat T)/\inn(T)\simeq\Mod(\hat\Si)$. 

Consider the group 
\[\Mod_\mu(\Si,*)=\{\phi\in\ca A(T): \phi(\hat T)=\hat T\text{ and } q\circ\phi=q\}.\] By definition, we have a homomorphism $\Mod_\mu(\Si,*)\ra\ca A(\hat T)$. Transferring from group theory to topology, we have the following diagram.  
\[\begin{xy}
(0,0)*+{1}="A";
(25,0)*+{\hat T}="B";
(50,0)*+{\Mod(\hat\Si,*)}="C";
(75,0)*+{\Mod(\hat\Si)}="D";
(100,0)*+{1}="E";
(0,-30)*+{1}="F";
(25,-30)*+{T}="G";
(50,-30)*+{\Mod(\Si,*)}="H";
(75,-30)*+{\Mod(\Si)}="I";
(100,-30)*+{1}="J";
(50,-15)*+{\Mod_\mu(\Si,*)}="K";
(15,-15)*+{\Mod_\mu(\Si,*)\cap T}="L";
{\ar"A";"B"}?*!/_3mm/{};
{\ar "B";"C"}?*!/^3mm/{};
{\ar "C";"D"}?*!/_3mm/{};
{\ar"D";"E"}?*!/_3mm/{};
{\ar"F";"G"}?*!/_3mm/{};
{\ar "G";"H"}?*!/^3mm/{};
{\ar "H";"I"}?*!/_3mm/{p};
{\ar"I";"J"}?*!/_3mm/{};
{\ar@{^{(}->}"L";"K"}?*!/_3mm/{};
{\ar "K";"C"}?*!/_3mm/{};
{\ar@{^{(}->}"K";"H"}?*!/_3mm/{};
{\ar@{-->}"K";"D"}?*!/^3mm/{r};
\end{xy}\]

$\Mod_\mu(\Si,*)<\ca A(T)\simeq\Mod(\Si,*)$ is finite index because $T$ has finitely many subgroups of index $|G|$ (permuted by $\ca A(T)$), the stabilizer of $\hat T$ acts on $G=T/\hat T$, and the group $\aut(G)$ is finite.

By the argument in \cite[Lemma 4.13]{morita_book}, there is a finite-index subgroup $\Ga<\Mod_\mu(\Si,*)$ so that $\Ga\cap T=\{e\}$, so $\Ga\hra\Mod(\Si)$; furthermore, since $\Mod_{\mu}(\Si,*)<\Mod(\Si,*)$ is finite index and $p$ is surjective, $\Ga$ is finite index in $\Mod(\Si)$. By construction we have finite-index subgroups $\Ga<\Mod_\mu(\Si)<\Mod(\Si)$ with a homomorphism $r:\Ga\ra\Mod(\hat\Si)^G$. Since $\Mod(\Si)\simeq\Mod(\bar S,\bar Z)$ and $\Mod(\hat\Si)^G\simeq\Mod(S,Z)^G$, this completes the proof. 
\end{proof}

{\it Remark.} The groups $\Mod(S,Z)^G$ and $\Mod_\mu(\bar S,\bar Z)$ have the same rational cohomology because of the exact sequence (\ref{eqn:ses}); this follows easily from examining the associated spectral sequence, since $H^*(G;\Q)$ is trivial. Our original interest in studying $H^*(\Sp^G;\Q)\ra H^*\big(\Mod(S)^G;\Q\big)$ was to determine if the isomorphism $H^*\big(\Mod_\mu(\bar S,\bar Z);\Q\big)\simeq H^*\big(\Mod(S,Z)^G;\Q\big)$ produced any cohomology in the cokernel of the injection $H^*(\Mod(\bar S,\bar Z);\Q)\hra H^*\big(\Mod_\mu(\bar S,\bar Z);\Q\big)$. Unfortunately this is not the case (at least stably) by Theorem \ref{thm:main}. 

\section{The index formula}\label{sec:index}

In this section $G\simeq\Z/m\Z$ and for concreteness we fix an identification $G<\C^\ti$.

The goal of this section is to prove Theorem \ref{thm:index} by deriving the index formula (\ref{eqn:index1}). To the author's knowledge, this derivation (of a families version of the $G$-index theorem for the signature operator) is not detailed in the literature, although it can be obtained by combining the contents of \cite{asi,asii,asiii,asiv}. Since these references are quite accessible, we will be brief and refer the reader to these papers for more detail. 

Let $M\xra\pi B$ be an $(S,G)$-bundle, and introduce a $G$-invariant fiberwise Riemannian metric. Denoting $S_b=\pi^{-1}(b)$ for $b\in B$, we have the de Rham complex $\Om_\C^*(S_b)$, its exterior derivative $d$, the adjoint $d^*$ of $d$ (defined using the Hodge star operator $\star$), and a self-adjoint elliptic operator $D=d+d^*$. The operator $\tau:\Om_\C^p\ra\Om_\C^{2-p}$ defined by $\tau=i^{p(p-1)+1}\star$ satisfies $\tau^2=1$ and $D\tau=-\tau D$, so $D$ restricts to operators $D^\pm:\Om^\pm\ra\Om^\mp$ on the $\pm1$ eigenspaces $\Om^\pm$ of $\tau$. The operators $D^+$ and $D^-$ are mutually adjoint, and $\ker(D^\pm)$ is the $\pm1$ eigenspace of $\tau$ acting on harmonic forms $\ca H^*(S_b;\C)$.

The collection $D_b^+: \Om^+(S_b)\ra\Om^-(S_b)$ for $b\in B$ defines a family of $G$-invariant differential operators on $S$. The \emph{(analytic) index} of the family $D=\{D_b^+\}$ is defined as
\[\ind(D)= E^+-E^-\in K_G(B),\]
where $E^\pm$ is the (equivariant K-theory class of the) bundle whose fiber over $b\in B$ is $\ca H^\pm(S_b;\C)$. 

The index theorem gives a topological description of the index: associated to $D$ is a \emph{symbol class} $\si\in K_G(T_\pi M)$, where $T_\pi M\ra M$ is the vertical (co)tangent bundle\footnote{The tangent and cotangent bundles are isomorphic.} of $\pi:M\ra B$ (and $K_G(\cdot)$ denotes equivariant K-theory with compact supports). In \cite{asiii} (see also \cite[pg.\ 40]{shanahan} and \cite[pgs.\ 236, 264]{lawmich}) it is shown that 
\[\si=\Om^+-\Om^-=(1+\bar L)-(1+L)=\bar L-L\in K_G(T_\pi M),\] where $L$ is the pullback of $T_\pi M\ra M$ along $T_\pi M\ra M$. The Thom isomorphism $K(M)\ra K(T_\pi M)$ is given by multiplication by the Thom class $u\in K(T_\pi M)$, and in this case $u=1-L$. Note that $(1+\bar L)(1-L)=\bar L-L$. Thus under $\pi_!:K(T_\pi^* M)\ra K(M)$ we have 
\[\si=\bar L-L=(1+\bar L)(1-L)\mapsto 1+T_\pi^*M.\]
The topological index is defined as $\tind=\pi_!(\si)$, where $\pi_!:K_G(M)\ra K_G(B)$ is the push-forward in K-theory. By the index theorem, $\ind(D^+)=\tind$, so 
\begin{equation}\label{eqn:indexformula}
\boxed{
E^+-E^-=\pi_!(1+T_\pi^*).}
\end{equation}

We want to understand (\ref{eqn:indexformula}) on the level of ordinary cohomology, i.e.\ under the map 
\begin{equation}\label{eqn:chernch}\ch_g:K_G(B)\simeq K(B)\ot R(G)\xra{1\ot\chi_g}K(B)\ot\C\xra{\ch} H^*(B)\ot\C,\end{equation}
where $R(G)$ is the representation ring and $\chi_g:R(G)\ra\C$ is the ring homomorphism that sends a representation $V$ to its character $\chi_g(V)$. 

The image of the left-hand side of (\ref{eqn:indexformula}) under (\ref{eqn:chernch}) is easily expressed. In $K(B)\ot R(G)$, we have 
\[E^+=\sum_{q^m=1} E_q\ot\rho_q\>\>\>\text{ and }\>\>\>E^-=\sum_q \bar E_q\ot\rho_{\bar q},\] 
where $\bar E_q$ denotes the conjugate bundle, $\bar q$ denotes the complex conjugate of $q\in\C$, and $\rho_q$ is the $\C G$ module $\C[x]/(x-q)$. It follows that   
\begin{equation}\label{eqn:lhs}\ch_g\big(E^+-E^-\big)=\sum_{q^m=1}\big(\ch(E_q)-\ch(\bar E_{\bar q})\big)\cdot\chi_g(\rho_q).\end{equation}
Observe that for $g=e^{2\pi i r/m}$, we have $\chi_{g}(\rho_q)=q^r$. 

In the remainder of the section we compute $\ch_g\big(\pi_!(1+T_\pi^*)\big)$. The pushforward $\pi_!$ in K-theory is hard to understand directly, so we want to commute $\ch$ and $\pi_!$, and compute $\pi_!$ in ordinary cohomology. This can be done after passing to the fixed point set $M^g$, using the Atiyah--Segal localization theorem \cite{asii}. 

\vspace{.1in}

{\it Atiyah--Segal localization theorem.} The character homomorphism $\chi_g:R(G)\ra\C$ factors through the localization $R(G)_{g}$ at the (prime) ideal $\ker(\chi_g)$. Denoting $K_G(M)_{g}$ the localization of the $R(G)$ module $K_G(M)$, there is a commutative diagram
\[\begin{xy}
(-20,0)*+{K_G(M)_g}="A";
(20,0)*+{K_G(M^g)_g}="B";
(-20,-20)*+{K_G(B)_g}="C";
(20,-20)*+{K_G(B)_g}="D";
{\ar"A";"B"}?*!/_3mm/{i^*/e};
{\ar "B";"D"}?*!/_3mm/{p_!};
{\ar "A";"C"}?*!/^3mm/{\pi_!};
{\ar@{=}"C";"D"}?*!/_3mm/{};
\end{xy}\]
Here $i^*$ is induced by $i:M^g\hra M$, the map $p$ is the restriction $\pi\rest{}{M^g}$, and $e=1-T_\pi^*\in K(M^g)$ is the Thom class (in K-theory) of the normal bundle of $M^g\hra M$. The top arrow is an isomorphism by \cite[Proposition 4.1]{segal_ekt}. Dividing by the Euler class makes the diagram commute (compare \cite[pg.\ 261]{lawmich}). 

In the diagram above, $\si=1+T_\pi^*\in K_G(M)_{(g)}$ maps to $\fr{1+T_\pi^*}{1-T_\pi^*}\in K_G(M^g)_{g}$. Now we can compute $\ch_g\circ p_!\left(\fr{1+T_\pi^*}{1-T_\pi^*}\right)$ using the following diagram 
\[\begin{xy}
(-20,0)*+{K_G(M^g)_g}="A";
(20,0)*+{K(M^g)\ot\C}="B";
(60,0)*+{H^*(M^g)\ot\C}="C";
(-20,-20)*+{K_G(B)_g}="D";
(20,-20)*+{K(B)\ot\C}="E";
(60,-20)*+{H^*(B)\ot\C}="F";
{\ar"A";"B"}?*!/_3mm/{\chi_g};
{\ar "B";"C"}?*!/_3mm/{\ch};
{\ar "D";"E"}?*!/_3mm/{\chi_g};
{\ar "E";"F"}?*!/_3mm/{\ch};
{\ar "A";"D"}?*!/_4mm/{p_!^K};
{\ar "B";"E"}?*!/_4mm/{p_!^K};
{\ar "C";"F"}?*!/_4mm/{p_!^H};
\end{xy}\]

If $g=1$, then the right square doesn't commute, but the failure to commute is the \emph{defect formula} $\ch\big[\pi_!\big(\si(D)\big)\big]=\pi_!^H\big[\ch\big(\si(D)\big)\cdot\Td(T_\pi)\big]$, where $\Td$ is the Todd class (this formula is also called the Grothendieck--Riemann--Roch computation). Since $\si(D)=1+T_\pi^*$, we have $\ch\big(\si(D)\big)=1+e^{-x}$, where $x=e(T_\pi M)$. Combining this with $\Td(T_\pi)=\fr{x}{1-e^{-x}}$ (see \cite[pg.\ 555]{asiii}), this recovers (\ref{eqn:indexfamilies}). 

If $g=e^{2\pi i r/m}\neq 1$, then the right square commutes because $p:M^g\ra B$ is a covering map. To express the class $\ch\circ \chi_g\left(\fr{1+T_\pi^*}{1-T_\pi^*}\right)$ in $H(M^g)\ot\C$, note that the character of $\chi_g(T_\pi^*)$ will vary on different components of $M^g$. Decompose $M^g=\sqcup M_j$, so that $e^{2\pi i/m}\in G$ acts on $T_\pi\rest{}{M_j}$ by rotation by $\ta_j=\fr{2\pi j}{m}$. Let $x_j$ denote the restriction of $e(T_\pi M)$ to $M_j$. Then, on $M_j$, we have
\[\ch\circ\chi_g\left(\fr{1+T_\pi^*\rest{}{M_j}}{1-T_\pi^*\rest{}{M_j}}\right)=\fr{1+e^{-i\>r\ta_j}e^{-x_j}}{1-e^{-i\>r\ta_j}e^{-x_j}}=\fr{e^{(x_j+i\>r\ta_j)/2}+e^{-(x_j+i\>r\ta)/2}}{e^{(x_j+i\>r\ta_j)/2}-e^{-(x_j+i\>r\ta_j)/2}}=\coth\left(\fr{x_j+i\>r\ta_j}{2}\right)\]

Combining these terms for all $j$, denoting $\ep_j=p_!(x_j)$, and combining with (\ref{eqn:lhs}) gives the desired index formula
\[\sum_{q^m=1}\big[\ch(E_q)-\ch(\bar E_{\bar q})\big]\cdot q^r=\sum_{\substack{1\le j\le m-1\\M_j\neq\es}}\coth\left(\fr{\ep_j+i\>r\ta_j}{2}\right).\]

{\it Remark.} 
We record here for later use the first two terms of the Taylor series of $\coth\left(\fr{x+i\vp}{2}\right)$ at $x=0$:
\begin{equation}\label{eqn:taylor}\coth\left(\fr{x+i\vp}{2}\right)\approx-i\cot(\vp/2)+\fr{1}{2}\csc^2(\vp/2)\>x.\end{equation}
Compare with \cite[pg.\ 150]{shanahan}. 

{\it Remarks.} 
\begin{enumerate}
\item The discussion above works generally when $S$ is replaced by a manifold of even dimension; for more details, see \cite{asiii} and \cite{erw}.
\item In the case $B=\pt$ and $G=\{e\}$ (i.e.\ the non-families, non-equivariant version of the index theorem), $\ind(D^+)\in  K(\pt)=\Z$ is equal to $\dim\ca H^+-\dim\ca H^-$, which is zero because the $\pm1$-eigenspaces of $\tau$ acting on $\ca H^1(S;\C)$ are conjugate (as complex vector spaces), so in particular they have the same dimension. However, in the families and/or equivariant case, $\ind(D^+)$ is nontrivial in general. \end{enumerate} 

\section{Computing $\al^*:H^2\big(\Sp^G;\Q\big)\ra H^2\big(\Mod^G;\Q\big)$}\label{sec:compute}

In this section we prove Theorem \ref{thm:main}. We proceed as follows. 
\begin{itemize} 
\item Step 1: We show $H^2(\Sp^G;\Q)\simeq\Q\{x_q: q^m=1,\>\im(q)\ge0\}$ using results of Borel \cite{borel_cohoarith}. For our computation, we require $S/G$ to have genus $h\ge6$. 
\item Step 2: We show that $\al^*(x_q)=c_1(\ca E_q)=c_1(\ca E_{\bar q})$ in $H^2(\Mod^G)$, where $\ca E=\bigoplus_{q^m=1}\ca E_q$ is the universal Hodge bundle over $B\Mod(S)^G$. This involves comparing two complex structures on the Hodge bundle. 

\item Step 3: Let $\si,\eta_{j_1},\ld,\eta_{j_n}\in H^2(\Mod^G)$ be the signature class and classes defined in Theorem \ref{thm:main}. The index formula gives a system of linear equations relating these classes to the classes $\al^*(x_q)$. Upon investigating this linear system, the result will follow from some character theory and a result about circulant matrices. 
\end{itemize} 

\subsection{The arithmetic group $\Sp^G$.}\label{sec:arithmetic}  In this section we compute $H^2(\Sp^G;\Q)$. This involves working out some of the general theory of arithmetic groups in the special case $\Sp^G$. Specifically, we (i) use restriction of scalars to show $\Sp^G$ is a lattice in a group $\bb G=\Sp_{2h}(\R)\ti\Sp_{2h'}(\R)\ti\prod_q\SU(a_q,b_q)$, (ii) use Borel--Matsushima to relate $H^j(\Sp^G;\Q)$ to the cohomology of a product of Grassmannians in some range $0\le j\le N$, and (iii) determine the range $N$ by giving a lower bound the $\Q$-rank of the irreducible factors of $\Sp^G$. To those familiar with arithmetic groups and their cohomology, (i) and (ii) are routine exercises. Our proof of (iii) uses the topology of the branched cover $S\ra S/G$ to find isotropic subspaces in sub-representations of $H_1(S;\Q)$. 

{\bf Restriction of scalars.} The group $\Sp^G$ acts by $G$-module maps on $H=H_1(S;\Q)$, so it preserves the decomposition $H=\bop_{k\mid m} H_k$ into isotypic components for the irreducible representations of $G$ over $\Q$. (Recall that the simple $\Q G$-modules are isomorphic to $\Q(\ze_k)$, where $\ze_k=e^{2\pi i/k}$ and $k\mid m$.) This induces a decomposition $\Sp^G=\prod_{k\mid m}\Ga_k$ into irreducible lattices. We want to identify $\Ga_k$ and determine the real semisimple Lie group $\bb G_k$ that contains $\Ga_k$ as a lattice. 

Fix $k\mid m$. For simplicity, denote $\ze=\ze_k$ and $\Ga=\Ga_k$. The representation $V=H_k$ is naturally a vector space over $\Q(\ze)$, and the intersection form $\om$ on $H$  determines a form $\be:V\ti V\ra\Q(\ze)$ given by
\begin{equation}\label{eqn:hermitian}\be(u,v)=-i\>\sum_{j=1}^k \om(u, t^jv)\cdot\ze^j.\end{equation}
Compare \cite[\S 3.1]{gllm}. If $k=1,2$, then $\be$ is symplectic, the group $\bb G=\Sp(V)$ preserving $\be$ is an algebraic group defined over $\Q$, and $\Ga\doteq\bb G(\Z)$ (commensurable). For $k\ge3$, $\be$ is Hermitian with respect to the involution $\tau(\ze)=\ze^{-1}$ on $\Q(\ze)$, the group $\bb G=\U(V,\be)$ of $\Q(\ze)$-linear automorphisms preserving $\be$ is an algebraic group $\bb G$ defined over $F=\Q(\ze+\ze^{-1})$ (the maximal real subfield of $\Q(\ze)$), and $\Ga\doteq\bb G(\ca O)$, where $\ca O\sbs F$ is the ring of integers. For a similar discussion, see \cite{looijenga_prym}. 

Restriction of scalars applied to $\bb G=\U(V,\be)$ gives an algebraic group $\bb G'$ defined over $\Q$ such that $\bb G'(\Z)$ is commensurable with $\bb G(\ca O)$. 
To define $\bb G'$, define an embedding $\si_q:F\ra\R$ by $\ze+\ze^{-1}\mapsto q+q^{-1}$ for each primitive $k$-th root of unity $q$ with $\im(q)>0$, and denote $\bb G^\si=\U(V,\si_q\circ\be)$. By the restriction of scalars construction, $\bb G'=\prod \bb G^{\si}$ is an algebraic group over $\Q$, the $\Z$-points $\bb G_\Z'$ is a lattice in $\bb G'$, and $\bb G_{\ca O}$ is commensurable with $\bb G_\Z'$. According to Witte-Morris \cite[Prop.\ 18.5.7]{morris}, the real points of $\bb G^{\si}$ is $\SU(a,b)$ for some $a,b\ge0$. 

Varying over all $k$, we find that  
\begin{equation}\label{eqn:rational}\Sp^G=\Sp(H_1)_\Z\ti \Sp(H_{2})_\Z\ti \prod_{\substack{k\mid m\\2<k\le\fr{m}{2}}}\U(H_k,\be_k)_{\ca O_k}\end{equation}
is a lattice in 
\begin{equation}\label{eqn:real}\Sp^G(\R)=\Sp_{2h}(\R)\ti\Sp_{2h'}(\R)\ti\prod_{\substack{q^m=1\\\im(q)>0}}\SU(a_q,b_q).\end{equation}
The second factor on the right-hand side of (\ref{eqn:rational}) and (\ref{eqn:real}) appears only when $m$ is even. 

{\it Remark.} In Section \ref{sec:app} we describe how to determine the integers $a_q,b_q$ using the Chevalley--Weil formula and the degree-0 term in the index formula (\ref{eqn:index1}). 

\vspace{.2in}
{\bf Borel--Matsushima.} In this section we recall the Borel--Matsushima description of $H^j(\Ga;\Q)$ when $\Ga=\Ga_k$ is an irreducible factor of $\Sp^G$ as in (\ref{eqn:rational}). In what follows we will only use the case $j=2$. For $\Ga\simeq\Sp_{2n}(\Z)$, it is well-known that $H^2(\Sp_{2n}(\Z);\Q)\simeq\Q$ when $n\ge3$ (see \cite[Theorem 3.4]{erw} or \cite[Theorem 5.3]{putman}). Thus we focus on the Hermitian case $k>2$. 

\begin{prop}\label{prop:latticecoho}
Let $\bb G$ be an algebraic group defined over a field $F$ whose associated real semisimple Lie group $\bb G(\R)$ is a product of unitary groups $\SU(a_q,b_q)$ for $q$ in some set $Q$. For a lattice $\Ga\doteq \bb G(\ca O_F)$, the map 
\begin{equation}\label{eqn:latticecoho}\vp:B\Ga\ra \prod B\SU(a_q,b_q)\sim\prod B\text{S}\big(\U(a_q)\ti \U(b_q)\big)\ra \prod B\U(a_q)\end{equation}
induces an isomorphism on $H^j(-;\Q)$ for $0\le j\le \big\lfloor(\rk_F(\Ga)-1)/2\rfloor$, where $\rk_F(\Ga)\equiv\rk_F(\bb G)$ is the $F$-rank.
\end{prop}

%
Focusing on degree 2, since $H^2\big(B\U(p);\Q\big)=\Q\{c_1\}$ for $p\ge1$, combining with the computation for $H^2(\Sp_{2h}(\Z);\Q)$ mentioned above, we have
\begin{cor}\label{cor:H2} Let $\Sp^G<\Sp^G(\R)$ be as in (\ref{eqn:rational}) and (\ref{eqn:real}). Assume $h,h'\ge3$ and $a_q,b_q\ge1$. If $2\le \min_{2<k\le m/2} \big\lfloor(\rk_{F_k}(\Ga_k)-1)/2\rfloor$, then
\begin{equation}\label{eqn:coho}H^2(\Sp^G;\Q)=\Q\{x_q: q^m=1,\>\im(q)\ge0\},\end{equation}
where $x_1$ and $x_{-1}$ are pulled back from $\Sp^G\ra\Sp_{2h}(\R)$ and $\Sp^G\ra\Sp_{2h'}(\R)$, respectively, and $x_q$ is pulled back from $\Sp^G\ra\SU(a_q,b_q)\sim\text{S}\big(\U(a_q)\ti\U(b_q)\big)\ra\U(a_q)$. \end{cor}

\begin{proof}[Proof of Proposition \ref{prop:latticecoho}]
The map on cohomology induced by (\ref{eqn:latticecoho}) can be realized more geometrically as follows (compare \cite[Proposition 7.5]{borel_cohoarith} and \cite[\S 3.2]{giansiracusa}). 
The cohomology $H^*(\Ga;\Q)$ can be identified with the cohomology of the complex $\Om^*(X)^\Ga$ of $\Ga$-invariant differential forms on the symmetric space $X=\bb G(\R)/K$, where $K<\bb G(\R)$ is a maximal compact subgroup. A first approximation to the cohomology of $\Om^*(X)^\Ga$ is the cohomology of the subalgebra $\Om^*(X)^{\bb G(\R)}$ of $G(\bb R)$-invariant forms, which can be identified with the cohomology $H^*(X_U;\Q)$ of the compact dual symmetric space 
\[X_U=\prod\SU(a_q+b_q)/\text{S}\big(\U(a_q)\ti\U(b_q)\big)\simeq\prod\Gr_{a_q}(\C^{a_q+b_q}).\] According to Borel \cite[Theorem\ 4.4(ii)]{borel_cohoarith2}, the inclusion $\Om^*(X)^{\bb G(\R)}\ra\Om^*(X)^\Ga$ induces an isomorphism $H^j(X_U;\R)\ra H^j(\Ga\bs X;\R)$ for $0\le j\le\min \big\{c(\bb G),\> m(\bb G(\R))\big\}$. In our case $c(\bb G)\ge \big\lfloor(\rk_F(\bb G)-1)/2\big\rfloor$ by \cite[\S9(3)]{borel_cohoarith}, and $m(\bb G(\R))\ge\rk_\R(\bb G(\bb R))/2$ by \cite[Theorem 2]{matsushima} (see also \cite[\S9.4]{borel_cohoarith}). Since $F$-rank is always less than or equal to $\R$-rank, we get an isomorphism for $0\le j\le \big\lfloor(\rk_F(\bb G)-1)/2\big\rfloor$. 

Furthermore, the obvious map $\Gr_a(\C^{a+b})\ra\Gr_a(\C^\infty)\simeq B\U(a)$ induces a map $X_U\ra\prod B\U(a_q)$ that is a cohomology isomorphism in degrees $0\le j\le 2\min b_q$. See \cite[Example 4.53]{hatcher}. Note that no $b_q$ can be smaller than $\min_k\{\rk_{F_k}(\Ga_k)\}$ because the $F$-rank for a unitary group is equal to the maximal dimension of an isotropic subspace \cite[Ch.\ 9]{morris}.

In summary the map $H^*\big(\prod B\U(a_q)\big)\ra H^*(X_U)\ra H^*(\Ga)$ induces an isomorphism in degrees $0\le j\le  \big\lfloor(\rk_F(\bb G)-1)/2\big\rfloor$, as desired. 
\end{proof}

{\bf $F$-rank and covers.} To apply Proposition \ref{prop:latticecoho}, we need to compute the $F$-rank of our lattice $\Ga_k<\U(H_k, \be_k)$, or at least bound it from below. 

\begin{prop}\label{prop:qrank}
Let $S$ be a surface with a $G=\Z/m\Z$ action, and let $h$ be the genus of $S/G$. Take $\Ga_k<\U(H_k,\be_k)$ as above (for any $k\mid m$, $k\ge3$). Then $\rk_{F_k}(\Ga_k)\ge h-1$.  
\end{prop}

\begin{proof}
By \cite[Ch.\ 9]{morris}, the $F_k$-rank of $\U(H_k,\be_k)$ is the maximal dimension of an $\be_k$-isotropic subspace of $H_k$ (as a vector space over $F_k$). By the definition of $\be_k$, to prove the proposition, it suffices to exhibit an $(h-1)$-dimensional $\om$-isotropic subspace of $H_1(S;\Q)$ (as a vector space over $\Q$). 

Denote $\bar S=S/G$ and let $\mu: S\ra\bar S$ be the quotient map. After removing the fixed points $Z=\fix(G)\sbs S$ and $\mu(Z)$, the map $\pi$ induces  a covering map $\Si\ra\bar\Si$ between surfaces with boundary. (Recall that we are assuming that $G$ acts trivially on $Z$.) Associated to this cover is a homomorphism $\rho: \pi_1(\bar\Si)\ra H_1(\bar\Si;\Z)\ra\Z/m\Z$. By Poincar\'e duality, there exists $[C]\in H_1(\bar\Si,\pa\bar\Si;\Z)$ so that $\rho(\ga)=[\ga]\cdot[C]\mod m$. Associated to the pair $(\bar\Si,\pa\bar\Si)$, we have an exact sequence
\[H_1(\bar\Si)\ra H_1(\bar\Si,\pa\bar\Si)\ra H_0(\pa\bar\Si),\]
so we can write $[C]=[A]+[B]$, where $A$ is a simple closed curve on $\bar\Si$ and $B$ is an arc connecting two boundary components of $\bar\Si$. Up to a change of coordinates, the pictures is as in the figure below. 
\begin{figure}[h!]
\labellist 
\small\hair 2pt 
\pinlabel $A$ at 105 440
\pinlabel $B$ at 220 500
\pinlabel $\bar\Si$ at 330 410
\endlabellist 
  \centering
\includegraphics[scale=.4]{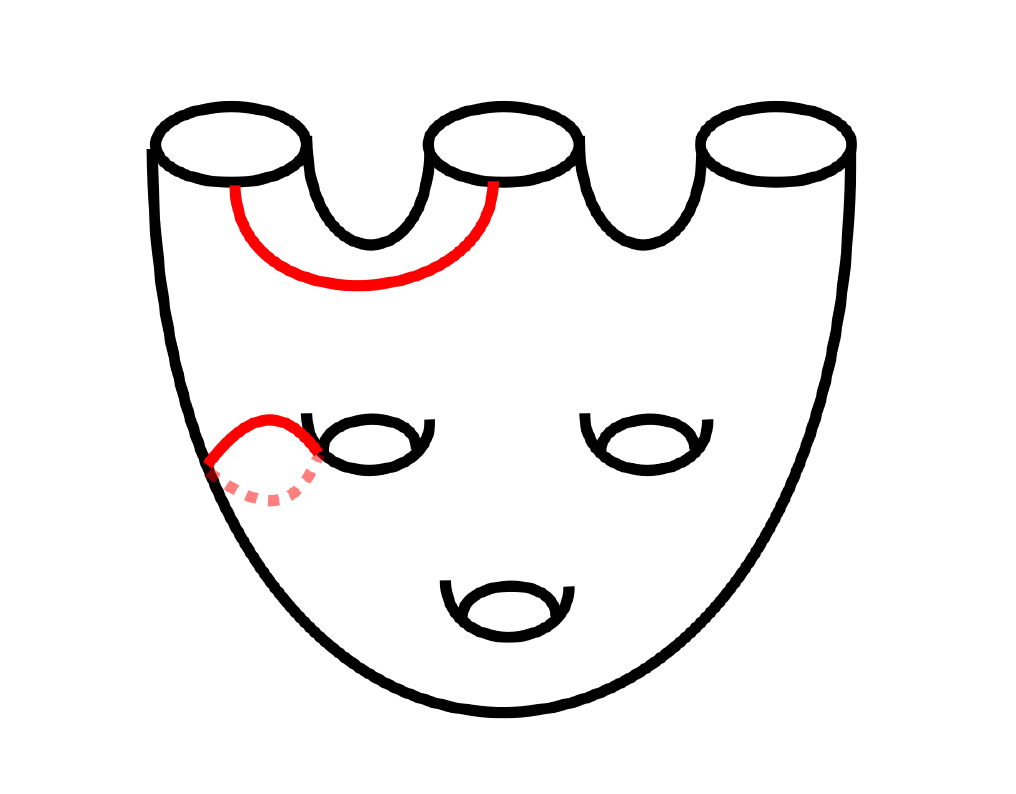}
\end{figure}

A simple closed curve $\ga\sbs \bar\Si$ lifts to $\Si$ if and only if $[\ga]\cdot [C]\equiv 0\mod m$. Thus we can find a genus-$(h-1)$ subsurface $N\sbs\bar\Si$ that lifts to $\Si\sbs S$, and gives an $(h-1)$-dimensional isotropic subspace of $H_1(S;\Q)$. 
\end{proof}

Since $\lfloor(\rk_{F_k}(\Ga_k)-1)/2\rfloor\ge\lfloor (h-2)/2\rfloor\ge2$ for $h\ge6$, this is the bound that appears in Theorem \ref{thm:main}. 


\subsection{Relating $H^*(\Sp^G;\Q)$ with Chern classes of the Hodge bundle.} 
In order to study the image of $\al^*:H^2(\Sp^G)\ra H^2(\Mod^G)$, we want to relate the classes $\al^*(x_q)\in H^2(\Mod^G)$ to the Chern classes $c_1(\ca E_q)\in H^2(\Mod^G)$, where $\ca E_q$ is the Hodge eigenbundle for the universal $(S,G)$ bundle. We will see that 
\begin{equation}\label{eqn:pullchern}c_1(\ca E_q)=\al^*(x_q)=c_1(\ca E_{\bar q})\end{equation}
for $q^m=1$ with $\im(q)>0$. This relation can be obtained by comparing two maps from $B\Mod^G$ to the product of unitary groups. For the first map, consider the composition
\begin{equation}\label{eqn:map1}B\Mod^G\ra B\Sp^G\ra B\Sp_{2g}(\R)\xra\sim B\U(g),\end{equation}
and note that it factors through $B\U(g)^G\ra B\U(g)$. The group $\U(g)^G$ is a product of unitary groups, one for each $q^m=1$. The map $B\Mod^G\ra B\Sp^G\ra B\U(g)^G$ classifies the Hodge eigenbundles (for the universal bundle).

The second map 
\begin{equation}\label{eqn:map2}B\Mod^G\ra B\Sp^G\ra B\Sp^G(\R)\sim B\U(h)\ti B\U(h')\ti\prod_{\substack{q^m=1\\\im(q)>0}}BS\big(\U(a_q)\ti\U(b_q)\big)\end{equation}
is obtained using (\ref{eqn:real}). 
On the bundle level this map is obtained by starting with an $(S,G)$ bundle $M\ra B$, taking the associated real vector bundle $H_1(S;\R)\ra E\ra B$, decomposing $E$ according to the decomposition of $H_1(S;\R)$ as a $G$-representation over $\R$, and giving this bundle a complex structure induced from the action of $G$ (the Hermitian forms (\ref{eqn:hermitian}) on sub-representations of $H_1(S;\Q)$ give $H_1(S;\R)$ a natural complex structure). 

The maps (\ref{eqn:map1}) and (\ref{eqn:map2}) classify the same bundle, but with respect to different complex structures. From this it follows that the terms in (\ref{eqn:pullchern}) differ by at most $-1$. The following proposition settles the difference. Although it suffices to work with the universal $(S,G)$ bundle, we find it more convenient to work on the level of individual bundles.




\begin{prop}\label{prop:chern}
Let $M\ra B$ be an $(S,G)$-bundle with Hodge bundle $E=\bigoplus_{q^m=1}E_q\ra B$. Then $c_1\big(E_q\big)=c_1\big(E_{\bar q}\big)=x_q$ in $H^2(B)$, where $x_q$ is defined in Corollary \ref{cor:H2}. 
\end{prop}

\begin{proof}[Proof of Proposition \ref{prop:chern}]
There are two natural complex structures on the bundle $H_1(S;\R)\ra E\ra B$ induced from different complex structures on $H_1(S;\R)$. The first $J$ is the Hodge star operator $\star^2=-1$ on $H^1(S;\R)$, and the second $J'$ is induced by the $G$ action on $H^1(S;\R)$ (this depends on a choice of generator of $G$). The proposition is proved by comparing $J$ and $J'$ and recalling how the definition of the Chern classes is sensitive to a choice of complex structure (see Borel--Hirzebruch \cite[\S9.1]{bh}).

Decompose $H^1(S;\R)=H(1)\op H(-1)\op\bigoplus_{\substack{q^m=1\\\im(q)>0}}H(q,\bar q)$ into isotypic components for the irreducible representations of $G$ over $\R$. (Recall that the simple $\R G$ modules are the trivial representation $V(1)$, the sign representation $V(-1)$ (if $m$ is even), and $V(q,\bar q)=\R[t]/\big(t^2-(q+\bar q)t+1\big)$ for $q^m=1$ such that $\im(q)>0$.)

The complex structure $J$ on $H^1(S;\R)$ induces an isomorphism $H(q,\bar q)\simeq H^1(S;\C)_q=H^{1,0}_q\op H^{0,1}_q$. This decomposition coincides with the decomposition of $H(q,\bar q)$ into positive-definite and negative-definite subspaces for the Hermitian form $\be$ in (\ref{eqn:hermitian}). Since $H^{1,0}$ and $H^{0,1}$ are $+i$ and $-i$ eigenspaces for $J$, the same holds for $H^{1,0}_q$ and $H^{0,1}_q$. This identifies the complex structure $J$ on these two factors. On the other hand, if we view $V=H^{1,0}_q\op H^{0,1}_q$ as a real vector space $V(q,\bar q)^N$, the $G$ action defines another complex structure $J'$ such that for any $v\in V$, the orientation $(v,J'v)$ on $\R\{v,J'v\}$ agrees with the orientation $(v,\tau v)$, where $\tau=e^{2\pi i/m}$ generates $G$. From this description, it follows that $J$ and $J'$ agree on $H^{1,0}_q$, but differ by $-1$ on $H^{0,1}_q$. 

Let $c_1(E_q)$ and $c_1'(E_q)$ denote the Chern class defined using $J$ and $J'$, respectively \cite[\S9.1]{bh}. 
Since $J=J'$ on $H^{1,0}_q$, we have $c_1(E_q)=c_1'(E_q)$ when $\im(q)\ge0$. Furthermore, $c_1'(E_q)=x_q$ by (\ref{eqn:coho}). When $\im(q)<0$, we have $c_1'\big(E_q\big)=-x_q$ because with respect to $J'$ the bundle $E_q\op E_{\bar q}\ra B$ is classified by the map $B\ra B\SU(a_q,b_q)\sim B\text{S}\big(\U(a_q)\ti\U(b_q)\big)$, which implies that $c_1'(E_q)=-c_1'(E_{\bar q})$. Since $J=-J'$ on $H^{0,1}_q$, we have $c_1(E_q)=-c_1'(E_q)=x_q$.
\end{proof}

\subsection{Applying the index formula.} The degree-1 terms of the index formulas (\ref{eqn:indexfamilies}) and (\ref{eqn:index1}) give a system of linear equations:

\begin{equation}\label{eqn:appind1}\sum_{\substack{q^m=1\\ \im(q)\ge0}}c_1(E_q)=\si/4,\end{equation}
and for $1\le r\le m-1$,
\begin{equation}\label{eqn:appind2}c_1(E_1)+(-1)^rc_1(E_{-1})+\sum_{\substack{q^m=1\\ \im(q)>0}} (q^r+\bar q^r)\> c_1(E_q)=\sum_{\substack{1\le j<m/2\\M_j\cup M_{m-j}\neq\es}} \csc^2(r\ta_j/2)\>\eta_j/4
\end{equation}
where $\eta_j$ is defined as in the statement of Theorem \ref{thm:main}, and the term with $c_1(E_{-1})$ appears only when $m$ is even. 

{\it Remark.} The reason we must consider $\eta_j$ is because $\csc^2(\cdot)$ is even, so the coefficient of $\ep_j$ and $\ep_{m-j}$ always agree when both are defined. 

Equations (\ref{eqn:appind1}) and (\ref{eqn:appind2}) define a matrix equation of the form 
\[J\>\left(\begin{array}{c}c_1(E_{\ze^0})\\c_1(E_{\ze^1})\\\vdots\\c_1(E_{\ze^d})\end{array}\right)=K\>\left(\begin{array}{c}\si\\\eta_{j_1}\\\vdots\\\eta_{j_n}\end{array}\right),\]
where $d=\lfloor m/2\rfloor$, $n$ is the number of $1\le j<m/2$ for which $M_j\cup M_{m-j}\neq\es$, $J$ is a $d\ti d$ matrix, and $K$ is a $d\ti(n+1)$ matrix. Note that $n\le\phi(m)/2$. 

We wish to show that $\im\big[H^2(\Sp^G)\ra H^2(\Mod^G)\big]=\Q\{\si,\eta_{j_1},\ld,\eta_{j_n}\}$. First we show that $J$ is invertible, which implies the containment $\subseteq$. Then we show $\rk(K)\ge n+1$, which implies the other containment. 

\begin{prop}
$J$ is invertible. 
\end{prop}

\begin{proof}
A column of $J$ has the form $\big(\chi_g(V_0)\>\>\chi_g(V_1)\>\>\cd\>\>\chi_g(V_d)\big)^T$ for fixed $g\in G\sbs\C^\ti$ with $\im(g)\ge0$, and where $V_j=\rho_{\ze^j}+\rho_{\ze^{-j}}$ for $1\le j<m/2$ and $V_j=\rho_{\ze^j}$ for $j=0, m/2$. 

If the columns are dependent, then there are constants $a_0,\ld,a_d$ so that 
\[a_0\>\chi_g(V_0)+\cd+a_d\>\chi_g(V_d)=0\]
for $g\in G\sbs\C^\ti$ with $\im(g)\ge0$. But then this equation holds for all $g\in G$ because $\chi_g(V_j)=\chi_{g^{-1}}(V_j)$. But this is impossible because the characters of irreducible representations of $G$ are linearly independent. 
\end{proof}

\begin{prop}
$\rk(K)\ge n+1$. 
\end{prop}

Using (\ref{eqn:appind1}) and (\ref{eqn:appind2}), note that $K=\left(\begin{array}{cc}1/4&0\\0&K'\end{array}\right)$, where $K'$ is an $(d-1)\ti n$ matrix. From inspection of (\ref{eqn:appind2}), to prove the proposition it suffices to show the following proposition. 

\begin{prop}
Fix $m\ge2$. Let $V\simeq\R^{\phi(m)/2}$ be a real vector space with basis $\{e_\ell\}$ for $1\le \ell<m/2$ and $\gcd(\ell,m)=1$. Then the vectors
\[v_k=\sum_{\substack{1\le \ell<m/2\\\gcd(\ell,m)=1}}\csc^2\left(\fr{\pi k\ell}{m}\right) e_\ell\]
$1\le k\le \phi(m)/2$ also form a basis for $V$. 
\end{prop}

\begin{proof}
We will denote $\Z/m\Z$ by $C_m$. For simplicity we start with the case $m=p$ is prime. The case $m=p^n$ is a prime power follows easily from this. Then we explain the general case. 

{\it Case 1:} $m=p$ is prime. Let $q=\fr{p-1}{2}$. Consider functions $f_k:(C_p)^\ti\ra\R$ defined by $f_k(x)=\csc^2\left(\fr{k\pi}{p}\>x\right)$, and $A=(A_{k,\ell})$ be the $q\times q$ matrix $A_{k,\ell}=f_k(\ell)$ for $1\le k,\ell\le q$. To prove the proposition, it is enough to show that $A$ is invertible. 

To this end, define another $q\times q$ matrix $B$ as follows. Consider the surjective homomorphism $\phi:(C_p)^\times\simeq C_{p-1}\rightarrow C_q$. For $0\le i,j\le q-1$, define $B_{ij}$ by $\csc^2\left(\frac{\pi}{p}\cdot y\right)$, where $\phi(y)=i+j$. This is well defined because $\csc^2(x)$ is an even function. 

Now observe
\begin{enumerate} 
\item $A$ and $B$ are the same matrix, up to permuting rows and columns. Thus it suffices to show that $\det(B)\neq0$. (We will show the eigenvalues of $B$ are all nonzero.) 

\item $B$ is a \emph{circulant matrix}, up to permuting rows and columns (see \cite{wiki:circulant} for the definition). This is easy to see because $B$ is obtained by taking the multiplication table for $\mathbb Z/q$ and applying a fixed function to each entry. (The multiplication table of a cyclic group is circulant up to permuting the rows.)
\end{enumerate} 

Now the eigenvalues/eigenvectors of a circulant matrix are easily computed \cite{wiki:circulant}. The eigenvalues have the form $\la_j=c_0+c_1\omega^j+c_2\omega^{2j}+\cdots+c_{q-1}\omega^{(q-1)j}$, where $\omega=e^{2\pi i/q}$ and $0\le j\le q-1$ and the $c_i$ are in bijection with $\csc^2\left(\frac{k\pi}{p}\right)$, $1\le k\le q$. If $\la_j=0$, then $\omega^j$ is a solution to the polynomial $P(x)=c_0+c_1x+\cdots+c_{q-1}x^{q-1}$ for some $j$. This is possible if and only if $c_0=c_1=\cdots=c_{q-1}$, which is not the case.

{\it Case 2:} $m=p^n$ is a prime power. An important feature of the above argument is that when $m$ is prime, $(C_m)^\ti\simeq C_{\phi(m)}$ is cyclic, as is $(C_m)^\ti/\{\pm1\}$, so its multiplication table is given by a circulant matrix whose determinant is easy to compute (even after applying a fixed function to each coordinate). 

When $p$ is an odd prime, then $(C_{p^n})^\ti\simeq C_{\phi(p^n)}$ is cyclic, so we may repeat the argument of Case 1. 

When $p=2$, the group $(C_{2^n})^\ti\simeq C_2\ti C_{2^{n-2}}$ is \emph{not} cyclic. However, the fact that $f_k$ is even implies that it factors through $(C_{2^n})^\ti/\{\pm1\}$, and the subgroup $\{\pm1\}<(C_{2^n})^\ti$ corresponds to the subgroup $C_2\ti\{0\}<C_2\ti C_{2^{n-2}}$. This means $f_k:C_2\ti C_{2^{n-2}}\ra\R$ factors though the cyclic group $C_{2^{n-2}}$, and we can again apply the argument from Case 1. 

{\it Case 3:} $m$ is arbitrary. In this case we cannot assume that $(C_m)^\ti$ is cyclic, and in most cases the multiplication table for $(C_m)^\ti/\{\pm1\}$ will not be circulant. However, if we write $m=p_1^{n_1}\cd p_r^{n_r}$, then using the isomorphism $(C_m)^\ti\simeq (C_{p_1^{n_1}})^\ti\ti\cd\ti (C_{p_r^{n_r}})^\ti$, the multiplication table for $(C_m)^\ti/\{\pm1\}$ may be expressed as a special kind of block circulant matrix. Having this block circulant form will allow us to apply the argument of Case 1 iteratively. 

We begin by examining what the structure of the multiplication table of a product of cyclic groups. Fix a finite group $F$ and a cyclic group $C_d=\lan t\ran$. If the multiplication table for $F$ is given by a matrix $A$, then the multiplication table for $F\ti C_d$ has the form
\[\left(\begin{array}{cccccccc}
A&tA&\cd&t^{d-1}A\\
tA&t^2A&\cd&A\\
\vdots&&\ddots&\vdots\\
t^{d-1}A&A&\cd&t^{d-2}A
\end{array}\right)\]
This matrix becomes block circulant after permuting the rows. Thus the multiplication table of a product of cyclic groups is an iterated block circulant matrix. 

Next we determine the eigenvalues of a block circulant matrix. Fix $d,n\ge1$, fix $A_0,\ld,A_{d-1}\in M_{n}(\R)$, and consider the block circulant matrix 
\[B=\left(\begin{array}{cccccccc}
A_0&A_1&\cd&A_{d-1}\\
A_{d-1}&A_0&\cd&A_{d-2}\\
\vdots&&\ddots&\vdots\\
A_1&A_2&\cd&A_0
\end{array}\right)\]
Suppose that the matrices $A_i$ share common eigenvectors $x_0,\ld,x_{n-1}$, so that $A_ix_j=\la_{ij}x_j$. Denoting $\ze=e^{2\pi i/d}$, the eigenvectors of $B$ are
\[x_{kj}=\left(\begin{array}{c}x_j\\\ze^k x_j\\\vdots\\\ze^{k(d-1)}x_j
\end{array}\right)\]
for $0\le k\le d-1$ and $0\le j\le n-1$, and the eigenvalues are 
\[\eta_{kj}=\la_{0k}+\la_{1k}\ze^j+\cd+\la_{m-1,k}\ze^{j(d-1)}\]
These facts are easily checked. 

Now the group $(C_m)^\ti/\{\pm1\}$ is a product of cyclic groups, so its multiplication table is an iterated block circulant matrix $B_0$. The matrix $A=\big(f_k(\ell)\big)$ is equivalent to the matrix $B$ obtained by applying $\csc^2(\fr{\pi}{m}\cdot)$ to each entry of $B_0$. Since all $n\ti n$ circulant matrices have the same eigenvectors, the above computation applies for computing the eigenvalues of $B$. Now, as in Case 1, the eigenvalues are given as degree $m-1$ polynomials $P(\exp^{2\pi i/m})$ with (nonconstant!) coefficients among the $f_k(\ell)$, so $\det (B)\neq0$. 
\end{proof}

Since $J$ is invertible and $\rk(K)\ge n+1$, we conclude that $H^2(\Sp^G)$ surjects to the subspace of $H^2(\Mod^G)$ generated by $\si$ and $\{\eta_j\}$, which finishes the proof of Theorem \ref{thm:main}.

\section{Further application of the index formula}\label{sec:app}

\subsection{The real points of $\Sp^G$.} We remark on how the degree-0 term of the index formula can be used to determine the group $\bb G$ that contains $\Sp^G$ as a lattice. This is an elaboration of a remark in \cite[\S3]{mcmullen_braid} and will be used later in this section. 

{\it Chevalley--Weil.} First one can use the Chevalley--Weil algorithm to determine the character $\chi_H$ of $H=H_1(S;\R)$. Obviously $\chi_H(e)=\dim H=2g$, and by the Lefschetz formula, $\chi_H(g)=2-\#\Fix(g)$ for $g\neq e$. Since a representation is determined by its character, this gives the integers $n_q$ in the decomposition 
\[H_1(S;\R)=V(1)^{n_1}\op \bigoplus_{\substack{q^m=1\\\im(q)>0}}V(q,\bar q)^{n_q}\op V(-1)^{n_{-1}}.\]
Here $V(\pm1)$ are the trivial/alternating representations, and $V(q,\bar q)=\R[t]/(t^2-(q+\bar q)t+1)$. 

{\it Hodge star and index formula.} The Hodge star gives a complex structure to $H_1(S;\R)$, and hence an isomorphism 
\[V(q,\bar q)^{n_q}\simeq V(q)^{a_q}\op V(\bar q)^{b_q}\]
for each $q$ with $\im(q)>0$, where $V(q)=\C[t]/(t-q)$. The numbers $a_q,b_q$ can be computed using the degree-0 term of the index formula (\ref{eqn:index1}) 
\begin{equation}\label{eqn:deg0}\sum_{q^m=1,\>\im(q)>0}(a_q-b_q)(q^r-\bar q^r)=-i\sum_{\substack{1\le j\le m-1\\Z_j\neq\es}}\cot(r\>\ta_j/2)\cdot|Z_j|,\end{equation}
where $Z_j$ is as in the statement of Theorem \ref{thm:index}. 

{\bf Example.} Let $G=\Z/m\Z$ act on a closed surface $S$ of genus $g=\fr{(m-1)(m-2)}{2}+mh$ with $m$ fixed points. These surfaces arise in Morita's $m$-construction \cite[\S4.3]{morita_book}. 

An explicit model for $S$ can be obtained as follows. Take $m$ disks, stacked horizontally, and attach $m$ strips between each pair of adjacent levels, as pictured in Figure \ref{fig:surf} (in the case $m=5$). This gives a surface of genus $\fr{(m-1)(m-2)}{2}$ with $m$ boundary components. The rotation by $2\pi/m$ on the disk extends to an action of $\Z/m\Z$ on this surface with one fixed point in each disk. Along each boundary component, we can attach a genus-$h$ surface (with one boundary component) to obtain a closed surface of genus $\fr{(m-1)(m-2)}{2}+mh$ with an action of $\Z/m\Z$ with $m$ fixed points. 

\begin{figure}[h!]
\labellist 
\small\hair 2pt 
\endlabellist 
  \centering
\includegraphics[scale=.6]{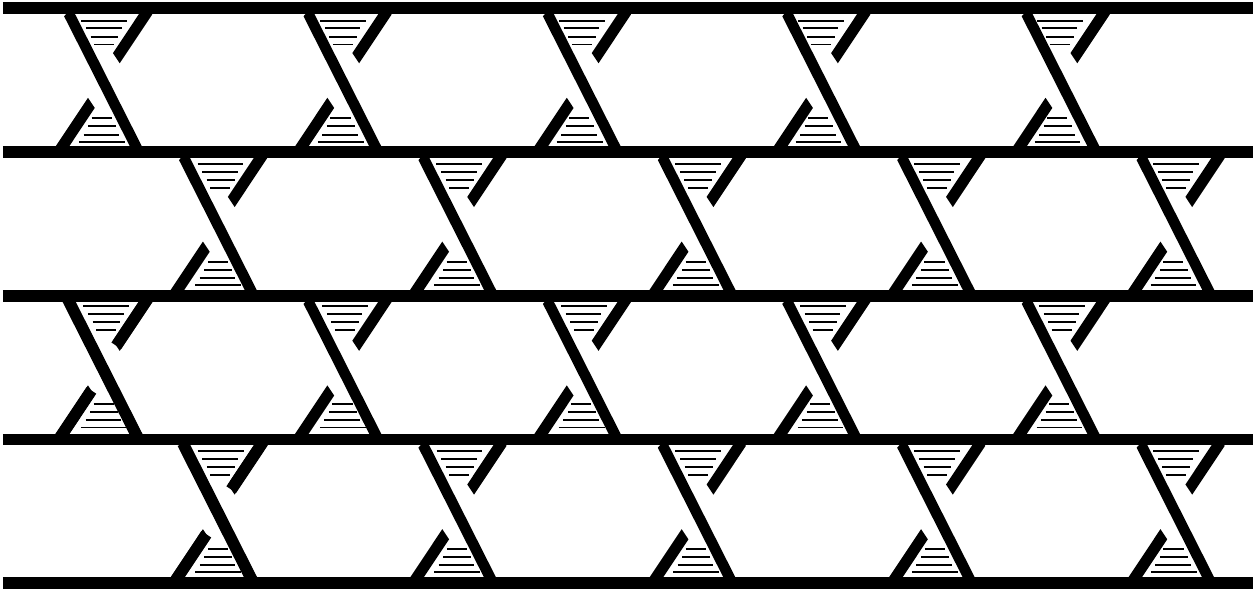}
      \caption{Schematic of a genus-6 surface with 5 boundary components and an action of $\Z/5\Z$.}
      \label{fig:surf}
\end{figure}

Using Chevalley--Weil, one easily computes 
\[H_1(S;\Q)=\Q^{2h}\op\bigoplus_{\substack{k\mid m\\k\ge2}}\Q(\ze_k)^{2h+m-2}.\]This is the decomposition $H_1(S;\Q)=\bigoplus_{k\mid m}H_k$ described in Section \ref{sec:arithmetic}. Applying (\ref{eqn:deg0}), we find that 
\[\prod_{\substack{k\mid m\\2<k\le m/2}}\U(H_k,\be_k)_{\ca O_k}\]
is a lattice in $\bb G(\R)=\prod_{i=0}^N\SU(h+i, h+m-2-i)$, where $N=\lfloor\fr{m-1}{2}\rfloor$. Equivalently, the factors in $\bb G(\R)$ are of the form $\SU(u+v_q)$ for $u,v_q\in\Z^2$, where $u=(h,h+m-2)$ and for each $q^m=1$ with $\im(q)>0$, we define $v_q=(a,-a)$ where $a$ the number of $m$-th roots of unity above the line from 1 to $q$ in $\C$. See Figure \ref{fig:hermitian} and also \cite[Figure 1]{mcmullen_braid}. 

\begin{figure}[h!]
\labellist 
\small\hair 2pt 
\pinlabel $q$ at 52 297
\pinlabel $1$ at 138 269
\endlabellist 
  \centering
\includegraphics[scale=.6]{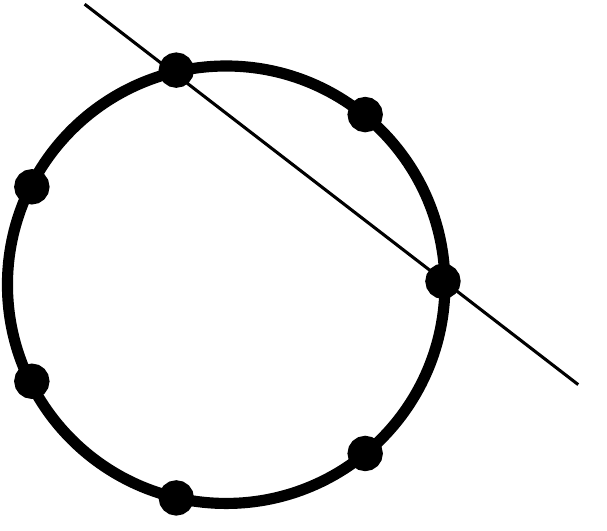}
      \caption{For $m=7$ the group $\bb G(\R)=\SU(h,h+5)\ti\SU(h+1,h+4)\ti\SU(h+2,h+3)$.}
      \label{fig:hermitian}
\end{figure}

\subsection{Relation to Hirzebruch's signature formula.}

Hirzebruch \cite{hirzebruch} explained how the signature changes in a branched cover. In this section we derive this result for surface bundles from our viewpoint. For simplicity we restrict to 2-fold branched covers. 

Let $M$ be a 4-manifold with a $G=\Z/2\Z$ action with fixed set $\Fix(G)=M_{0}$. In this case, Hirzebruch proved that 
\begin{equation}\label{eqn:hirz}\si(M)=2\>\si(M/G)-\si(M_0\cdot M_0),\end{equation}
where $M_0\cdot M_0$ is the self-intersection (which is a homology class; its signature is well defined). This formula applies in the special case when $M$ is the total space of an $(S,G)$-bundle over a surface. Our main observation here is that the terms $\si(M/G)$, $\si(M_0\cdot M_0)$ can be understood in terms of cohomology of the arithmetic group $\Sp^G=\Sp_{2h}(\Z)\ti\Sp_{2h'}(\Z)$. 

To illustrate this, consider a $G$ action on a genus-$2h$ surface with two fixed points $Z=\{z_1,z_2\}$. The quotient $\mu:S\ra\bar S$ has genus $h$. Let $\bar Z=\mu(Z)$. In this case $\Sp^G=\Sp_{2h}(\Z)\ti\Sp_{2h}(\Z)$, and we have a commutative diagram
\[\begin{xy}
(20,0)*+{\Mod(S,Z)^G}="B";
(60,0)*+{\Sp^G}="E";
(20,-20)*+{\Mod_\mu(\bar S,\bar Z)}="D";
(60,-20)*+{\Sp_{2h}(\Z)}="F";
{\ar "B";"D"}?*!/^3mm/{\phi};
{\ar"B";"E"}?*!/_3mm/{f};
{\ar"D";"F"}?*!/_3mm/{g};
{\ar"E";"F"}?*!/_3mm/{\psi};
\end{xy}\]
The cohomology $H^2\big(\Sp^G;\Q\big)$ is generated by $\{x_1,x_{-1}\}$ as in Corollary \ref{cor:H2}. Also $H^2\big(\Sp_{2h}(\Z);\Q\big)\simeq\Q\{y_1\}$ for $h\ge3$. Let $\si$ and $\bar\si$ be the signature classes in the cohomology of $\Mod(S,Z)^G$ and $\Mod_\mu(\bar S,\bar Z)$, respectively, and let $e_i\in H^2\big(\Mod(S,Z)^G\big)$ be the Euler class at the fixed point $z_i$ for $i=1,2$.

By the non-equivariant version of the index theorem $g^*(y_1)=\bar\si/4$. Hirzebruch's formula will come from determining $\phi^*(\bar\si)$. Since the diagram commutes and $\psi^*(y_1)=x_1$, we want to compute  $f^*(x_1)$. By the index formulas (\ref{eqn:appind1}) and (\ref{eqn:appind2}), we have $f^*(x_1+x_{-1})=\si/4$ and $f^*(x_1-x_{-1})=(e_1+e_2)/4$, so 
\begin{equation}\label{eqn:hirz2}\phi^*(\bar\si)=f^*(4\>x_1)=\fr{\si+e}{2},\end{equation} where $e=e_1+e_2$.

To conclude, let $M\ra B$ be an $S$ bundle with $B$ a surface and with monodromy $\phi:\pi_1(B)\ra \Mod(S,Z)^G$. The fixed set $M_0=\Fix(G)$ is a surface and $\phi^*(e)\in H^2(B)$ measures its self-intersection. Then (\ref{eqn:hirz2}) gives \[\sig(M/G)=\fr{\sig(M)}{2}+\fr{\#(M_0\cdot M_0)}{2}.\]
Since the signature of a 0-manifold is the number of points, this is the same as Hirzebruch's formula (\ref{eqn:hirz}).

\subsection{Toledo invariants of surface group representations.} The Toledo invariant $\tau$ is an integer invariant of a representation $\al :\pi_1(\Si)\ra G$, where $\Si$ is a closed oriented surface (genus $\ge1$) and $G$ is a \emph{Hermitian} Lie group. In this section we will be interested in the case $G=\SU(p,q)$ with $1\le p\le q$. To define $\tau(\al )$, first construct a smooth $\al $-equivariant map $f:\wtil\Si\ra X$, where $\wtil\Si$ is the universal cover of $\Si$ and $X=\SU(p,q)/S\big(\U(p)\ti\U(q)\big)$ is the symmetric space associated to $G$. The \emph{Toledo invariant} is defined as 
\[\tau(\al )=\fr{1}{2\pi}\int_Ff^*\om,\]
where $\om$ is the K\"ahler form of $X$ and $F\sbs\wtil\Si$ a fundamental domain for the action of $\pi_1(\Si)$. 

Domic--Toledo \cite{domic-toledo} showed that $|\tau(\al )|\le -p\>\chi(\Si)$, and Bradlow--Garcia-Prada--Gothen \cite{bgpg} have shown that components of the representation variety $\Hom\big(\pi_1(\Si),G\big)/G$ are in bijection with values of achieved by $\tau$. Here we simply observe that the Atiyah--Kodaira construction gives examples of surface group representations whose Toledo invariant can be computed using the index formula. 

We'll explain this in a special case (see \cite[\S4.3]{morita_book} for a general discussion of the Atiyah--Kodaira construction). Let $G=\lan\tau\ran\simeq \Z/7\Z$ and let $\bar S=S_h$ be a closed surface with a free $\Z/7\Z$ action. The product bundle $\bar S\ti\bar S\ra\bar S$ admits 7 disjoint sections $\Ga_{\bbm1},\Ga_{\tau},\ld,\Ga_{\tau^6}$, where $\Ga_f$ denotes the graph of $f:\bar S\ra\bar S$. In order to branch over $\bigcup \Ga_{\tau^i}$, we must first pass to a cover. Let $p:\Si\ti \bar S$ be the $\Z/7\Z$ homology cover ($\Si$ has genus $7^{2h}(h-1)+1$). The bundle $\Si\ti\bar S$ has sections $\Ga_{p},\Ga_{\tau p},\ld,\Ga_{\tau^6p}$, and admits a $\Z/7\Z$ branched cover $M\ra \Si\ti\bar S$ with branching locus $\bigcup\Ga_{\tau^ip}$. Projecting $M\ra\Si\ti\bar S\ra\Si$ defines a bundle with fiber $S$, which is a 7-fold branched cover $\mu:S\ra\bar S$ branched along 7 points ($S$ has genus $7h+15$). The homology $H_1(S;\Q)$ is isomorphic to $\Q^{2h}\op\Q(\ze_7)^{2h+5}$ as a $G$-module. In this case, $\Sp^G=\Sp_{2g}(\Z)\ti\Ga$, where $\Ga$ is an irreducible lattice in $\SU(h,h+5)\ti\SU(h+1,h+4)\ti\SU(h+2,h+3)$. Thus we have a map 
\[\al :\pi_1(\Si)\ra\SU(h,h+5)\ti\SU(h+1,h+4)\ti\SU(h+2,h+3).\]
Let $\al _i$ be the representation obtained by projecting to the $i$-th factor, $i=1,2,3$. By the index formula, the Toledo invariants are given by 
\[\tau(\al _1)=\fr{3}{112}\si,\>\>\>\>\>\>\tau(\al _2)=\fr{5}{112}\si,\>\>\>\>\>\>\tau(\al _3)=\fr{6}{112}\si,\]
where $\si=\sig(M)$. 

{\it Remark.} The Toledo invariants of representations obtained in this way will never have maximal Toledo invariant. This is because the Gromov norm of the Toledo class decreases when pulled back the mapping class group \cite{kotschick}, so in fact, no representation $\pi_1(\Si)\ra\SU(p,q)$ that factors through $\Mod(S)$ will be maximal. However, one could also ask whether these representations are maximal among representation that factor though $\Mod(S)$. 

\subsection{Cobordism invariants.}

Church--Farb--Thibault \cite{cft} show that the odd MMM classes $\kappa_{2i-1}$ are \emph{cobordism invariants}. This means that for an $S$ bundle $M^{4i}\ra B$, the characteristic number $\kappa_{2i-1}^{\#}(M\ra B)$ depends only on the cobordism class of $M$. In particular, the class $\kappa_{2i-1}$ cannot distinguish between different fiberings of a $4i$-manifold $M$. 

If $M\ra B$ admits a fiberwise $G$-action, we can ask about characteristic classes $c$ that are \emph{$G$-cobordism invariants}, i.e.\ the corresponding characteristic number $c^{\#}(M\ra B)$ depends only on the \emph{$G$-bordism class} of $M$ (for more on the notion of $G$-bordsim, see e.g.\ \cite[Chapter III]{conner-floyd}). Consider the case $\dim(M)=4$. Of course $\kappa_1^{\#}(M\ra B)$ is also a $G$-cobordism invariant; below we prove Corollary \ref{cor:bord} thus exhibiting more classes that have this property. 

\begin{proof}[Proof of Corollary \ref{cor:bord}]
Let $\Si$ be a closed surface and fix an $(S,G)$ bundle $M^4\ra\Si$. Let $E\ra\Si$ be the Hodge bundle with eigenbundles $E=\bigoplus_{q^m=1}E_q$. We aim to show that the numbers
\[c_1^{\#}(E_q\ra\Si)=\big\lan c_1(E_q),[\Si]\big\ran\]
depend only on the $G$-bordism class of $M$. 

Suppose that there is a $G$-manifold $W^5$ such that $M=\pa W$ (as $G$-manifolds). To prove the corollary, we must show that $c_1^{\#}(E_q\ra\Si)=0$. First observe that, by Theorem \ref{thm:main}, $c_1^{\#}(E_q\ra\Si)$ is a linear combination of the signature $\sig(M)$ and the intersection numbers $\#(M^\tau_j\cdot M^\tau_j)$, where $\tau$ generates $\Z/m\Z$, and we decompose the fixed set $M^\tau=\bigcup_{j=1}^{m-1}M^\tau_j$ according to the action of $\tau$ on the normal bundle (as in the statement of Theorem \ref{thm:index}). Now $\sig(M)=0$ because $M=\pa W$, and we claim that $\#(M^\tau_j\cdot M^\tau_j)=0$ as well. To see the latter, note that $M^\tau$ and $W^\tau$ are submanifolds (average a metric so that $\tau$ acts by isometries), and $M^\tau=\pa(W^\tau)$ because $M=\pa W$ as $G$-manifolds. It follows that $M^\tau\cdot M^\tau=\pa(W^\tau\cdot W^\tau)$. Since $W^\tau$ is a 3-manifold, $W^\tau\cdot W^\tau$ is a 1-manifold with boundary, and the boundary points occur in pairs, which implies that  $\#(M^\tau\cdot M^\tau)=0$, as desired. 
\end{proof}

{\it Remark.} It would be interesting to determine precisely which elements of $H^*\big(\Mod(S)^G;\Q\big)$ are $G$-cobordism invariants following Church--Crossley--Giansiracusa \cite{ccg}.

\bibliographystyle{plain}
\bibliography{modarith-bib}

\end{document}